\documentclass[12pt]{article}
\usepackage{amsmath,amsfonts,amssymb,amsthm} \usepackage{epsfig}

\theoremstyle{plain}
\newtheorem{theorem}{Theorem}[section]

\newtheorem{lemma}[theorem]{Lemma}
\newtheorem{proposition}[theorem]{Proposition}

\theoremstyle{definition}

\newtheorem{remark}[theorem]{Remark}

\numberwithin{equation}{section}

\newcommand{\Z}{{\mathbb Z}}
\newcommand{\R}{{\mathbb R}}

\newcommand{\barint}{\mathop{\hbox{\vrule height3pt depth-2.7pt
width.65em}\hskip-1em \int}}

\DeclareMathOperator{\rad}{rad}

\DeclareMathOperator{\supp}{supp}

\DeclareMathOperator{\BMO}{BMO}

\begin{document}
\title{A characterization of $\BMO$ self-maps of a metric measure space}

\author{Juha Kinnunen \and Riikka Korte \and Niko Marola \and
  Nageswari Shanmugalingam \footnote{The research is supported by the Academy of Finland.}}

\maketitle

\begin{abstract} 
  This paper studies functions of bounded mean oscillation ($\BMO$) on metric
  spaces equipped with a doubling measure. The main result gives characterizations for
  mappings that preserve $\BMO$. 
   This extends the corresponding Euclidean results by Gotoh to metric measure spaces.
  The argument is based on a generalizations Uchiyama's
  construction of certain extremal $\BMO$-functions and John-Nirenberg's lemma. 
  \end{abstract}

\medskip

{\small \emph{Mathematics Subject Classification (2010)}: Primary
  30L99; Secondary 43A85.}

{\small \emph{Keywords}: Bounded mean oscillation, doubling condition, John--Nirenberg lemma, analysis on metric measure spaces.}

\section{Introduction}
Let $X$ be a complete metric space equipped with a
metric $d$ and a Borel regular outer measure $\mu$ satisfying the
doubling condition. 
A locally integrable function $f:X\rightarrow \R$ is of
bounded mean oscillation, denoted as $f\in \BMO(X)$, if
\[
\|f\|_*=\sup\barint_{B}|f-f_{B}|\,d\mu<\infty,
\]
where the supremum is taken over all balls $B\subset X$.
We discuss invariance properties of $\BMO$-functions. 
More precisely, we extend a characterization of Gotoh \cite{Gotoh01, Gotoh05} of
mappings that preserve $\BMO$ to the metric setting. 
A $\mu$-measurable map $F\colon X \to X$ is a $\BMO$-map if 
$F^{-1}(E)$ is a $\mu$-null set for each $\mu$-null set $E\subset X$, 
for every $f\in\BMO(X)$ the composed map $C_F(f)=f\circ F$ is in $\BMO(X)$.
The first condition guarantees the uniqueness of the $\BMO$-map. 
Moreover, the composition operator $C_F$ is a bounded operator from $\BMO(X)$ to $\BMO(X)$.

The class of $\BMO$-functions is used, for example,  in harmonic analysis, partial differential equations and quasiconformal mappings.
Indeed, the first invariance property for $\BMO$-functions was obtained by Reimann  \cite{Rei}, where he showed that a homeomorphism is 
a $\BMO$-map if and only if it is quasiconformal, provided the homeomorphism is assumed to be differentiable almost everywhere. Later Astala
showed in \cite{Ast} that the differentiability assumption is superfluous for a suitably localized result. 
The advantage of the approach by  Gotoh \cite{Gotoh01} is that it applies to general measurable functions and hence is a more suitable to
extensions to the metric setting.
The Euclidean theory for $\BMO$-functions is well understood, but not so much in a general metric measure space. 
For related metric space results we refer to \cite{Buckley, Kronz, Maasalo, MMNO} and also to \cite[Section 3.3]{Bjornsbook}.

We generalize the construction of certain extremal $\BMO$-functions by
Uchiyama~\cite{Uchiyama82} (see also \cite[Section 2]{GJ}) to doubling
spaces. The result is stated in Theorem~\ref{theorem:construction} and
it constitutes the first part of the present paper.
In the second part, we consider characterizations of $\BMO$-maps
between doubling spaces. Our main result is stated in
Theorem~\ref{thm:Gotoh}. The characterizations in
Theorem~\ref{thm:Gotoh} are along the lines of the ones due to
Gotoh~\cite{Gotoh01,Gotoh05}.

\section{Construction of certain $\BMO$-functions}
\label{sect:U}

Throughout the paper, $X$ is a complete metric space equipped with a
metric $d$ and a Borel regular outer measure $\mu$ satisfying the
doubling condition. An open ball
\[
B(x,r)=\{y\in X:d(y,x)<r\},\quad x\in X,\,r>0,
\]
is simply denoted by $B$, we write $\rad(B)$ for the radius of the
ball $B$, and $\lambda B=\{y\in X:d(y,x)<\lambda r\}$, $\lambda>0$, is
the ball with the same center, but the radius dilated by the factor
$\lambda$.

In this paper, the doubling condition means that there exists a
constant $c_{D}>1$ such that for all $x\in X$, $0<r<\infty$ and $y\in
X$ such that $B(x,2r)\cap B(y,r)\neq \emptyset$, we have
\[
\mu(B(x,2r))\leq c_{D}\mu(B(y,r)).
\]
Notice that this condition is usually required to hold only for $x=y$,
but if this standard doubling condition is valid with some uniform
constant $c_{\mu}$, then $\mu(B(x,2r))\leq \mu(B(y,8r))\leq
c_{\mu}^{3}\mu(B(y,r))$, i.e. our version of the standard doubling
condition is satisfied with $c_{D}=c_{\mu}^{3}$.
The standard doubling condition implies that if $B(x,R)\subset X$,
$y\in B(x,R)$, and $0<r\leq R<\infty$, then
\[
\frac{\mu(B(y,r))}{\mu(B(x,R))} \geq
c_\mu^{-2}\left(\frac{r}{R}\right)^{\log_2c_\mu}.
\]
We refer, for instance, to \cite[Lemma 3.3]{Bjornsbook}.

We recall that a locally integrable function $f:X\rightarrow \R$ has
bounded mean oscillation, denoted as $f\in \BMO(X)$, if
\[
\|f\|_*=\sup\barint_{B}|f-f_{B}|\,d\mu<\infty,
\]
where the supremum is taken over all balls $B\subset X$.
We will identify functions which only differ by a constant; we shall
call $\|f\|_*$ the $\BMO$-norm of $f$. Here both $f_B$
and the barred integral $\barint_{B}f\,d\mu$ denote the integral average of
$f$ over a ball $B$.

The following theorem is a metric space counterpart of a construction
of certain $\BMO$-functions in Uchiyama~\cite{Uchiyama82} and Garnett--Jones~\cite{GJ}.

\begin{theorem}\label{theorem:construction}
  Let $\lambda>1$ and let $E_{1},\ldots,E_{N}$, $N\ge2$, be $\mu$-measurable subsets of
  $X$ such that
\begin{equation}\label{eqn:minEj}
\min_{1\leq j\leq N}\frac{\mu(E_{j}\cap B)}{\mu(B)}\leq c_{D}^{-4\lambda}
\end{equation}
for any ball $B\subset X$.
Then there exist functions $\{f_{j}\}_{j=1}^{N}$ such that
\begin{equation}\label{eqn:sum1}
\sum_{j=1}^{N}f_{j}(x)=1,
\end{equation}
and for each $1\leq j\leq N$
\begin{equation}\label{eqn:leq1}
0\leq f_{j}(x)\leq 1,
\end{equation}
and
\begin{equation}\label{eqn:E0}
  f_{j}(x)=0\quad\text{$\mu$-almost everywhere on }E_{j},
\end{equation}
and moreover,
\begin{equation}\label{eqn:bmoF}
\|f_{j}\|_*\leq \frac{c_{1}}\lambda.
\end{equation}
Here  $c_{1}$ is a constant that only depends on $c_{D}$ and $N$.
Conversely, if there exists $\{f_{j}\}_{j=1}^{N}$ that satisfy
\eqref{eqn:sum1}--\eqref{eqn:E0} and
\[
\|f_{j}\|_*\leq \frac{c_{2}}\lambda
\]
holds with a sufficiently small constant $c_{2}$, only depending on
$c_{D}$ and $N$, for every $1\leq j\leq N$, then \eqref{eqn:minEj}
holds.
\end{theorem}

Before the proof of the theorem, we fix some notation and state few
lemmas that will be needed later.
Let $q$ be a large integer, depending only on $c_{D}$ and $N$, such
that
\begin{equation}\label{eqn:q}
  1+N c_{D}^{6}q\leq 2^{q}.
\end{equation}
For every $k\in\Z$, let $r_{k}=2^{-kq}$ and let $\mathcal D_{k}$ be a
maximal set of points such that $d(x,y)\geq \tfrac12r_{k}$ whenever
$x,y\in \mathcal D_{k}$. Let $\mathcal D=\bigcup_{k\in \Z}\mathcal
D_{k}$. Moreover, let
\[
\mathcal B_{k}=\{B(x,r_{k})\,:\, x\in \mathcal D_{k}\}.
\]
From the maximality of the set $\mathcal D_{k}$ it follows
that for every $k\in\mathbb Z$,
\[
X=\bigcup_{B\in \mathcal B_{k}}B.
\]
We say that a function $a\in C(X)$ is adapted to a ball $B=B(x,r)$, if
\[
\supp a\subset B(x,2r)
\quad
\text{and}
\quad
|a(x)-a(y)|\leq \frac{d(x,y)}r.
\]
For a ball $B$, we set
\begin{equation} \label{g}
g_{j}(B)=\log_{c_{D}}\frac{\mu(B)}{\mu(E_{j}\cap B)},\quad 1\leq j\leq N.
\end{equation}

Let us state the following simple lemma for the function $g_j$.

\begin{lemma}\label{lemma:g}
Let $k$  be a positive integer.
If $B_{1}\subset B_{2}$ and $c_{D}^{k}\mu(B_{1})\geq \mu(B_{2})$ for the balls $B_1$ and $B_2$ in $X$, then 
\[
g_{j}(B_{1})\geq g_{j}(B_{2})-k.
\]
\end{lemma}
\begin{proof} Clearly
\[
\begin{split}
g_{j}(B_{1})&=\log_{c_{D}}\frac{\mu(B_{1})}{\mu(B_{1}\cap E_{j})} \\
&\geq\log_{c_{D}}\frac{c_{D}^{-k}\mu(B_{2})}{\mu(B_{2}\cap E_{j})}=g_{j}(B_{2})-k.\qedhere
\end{split}
\]
\end{proof}

The next result is well known for the experts, but we recall it here.

\begin{lemma}\label{lemma:coifman-weiss}
Let $f\in\BMO(X)$. Then
\[
\tfrac12\|f\|_*\leq\sup\left|\int_{X}fg\,d\mu\right|\leq\|f\|_*,
\]
where the supremum is taken over all functions $g$ for which there
exists a ball $B$ such that
\[
\supp g\subset B,\quad \|g\|_{\infty}\leq
\frac1{\mu(B)},\quad\text{and}\quad\int_{X}g\, d\mu=0.
\]
Conversely, if $f$ is a locally integrable function on $X$ and the
supremum above is finite, then $f\in\BMO(X)$ with the above norm
estimate.
\end{lemma}

\begin{proof}
First notice that for any $g$ as above, we have
\[
\left|\int_{X}fg \,d\mu\right|
=\left|\int_{X}(f-f_{B})g\,d\mu\right|\leq\barint_{B}|f-f_{B}|\,d\mu\leq
\|f\|_*.
\]
This gives the upper bound.

To see the lower bound, let $\varepsilon>0$ and let $B$ be a ball such
that
\[
\|f\|_*\leq \barint_{B}|f-f_{B}|\,d\mu+\varepsilon.
\]
Let $h\in L^\infty(B)$ with $\|h\|_{L^\infty(B)}\le1$ be a function
for which
\begin{equation} \label{eq:lemma}
\int_{B}|f-f_{B}|\,d\mu
=\int_{B}(f-f_{B})h\,d\mu.
\end{equation}
Since $\int_{B}(f-f_{B})\,d\mu=0$, we have
\begin{equation} \label{eq:coifman-weiss1}
\int_{B}|f-f_{B}|\,d\mu
=\int_{B}(f-f_{B})(h-h_B)\,d\mu.
\end{equation}
Define
\[
g=\frac{(h-h_B)\chi_B}{2\mu(B)}.
\]
Then 
\[
\supp g\subset B, \quad \|g\|_{L^\infty(B)}\le \frac1{\mu(B)} \quad \text{and}\quad \int_Xg\,d\mu=0. 
\]
Moreover
\begin{align} \label{eq:coifman-weiss2}
\int_Xfg\,d\mu & =\frac1{2\mu(B)}\int_Bf(h-h_B)\,d\mu \nonumber \\
& =\frac1{2\mu(B)}\int_B(f-f_B)(h-h_B)\,d\mu.
\end{align}
By combining the equation \eqref{eq:coifman-weiss1} and
\eqref{eq:coifman-weiss2} we conclude that
\[
\int_{B}|f-f_{B}|\,d\mu
=2\mu(B)\int_Xfg\,d\mu
\]
and
\[
\int_{X}fg\,d\mu = \frac12\barint_{B}|f-f_{B}|\,d\mu\geq\frac12(\|f\|_*-\varepsilon).
\]
The claim follows by passing $\varepsilon\to0$.

The equation \eqref{eq:lemma} together with the above inequalities
also indicates that the finiteness of
$\sup\left|\int_{X}fg\,d\mu\right|$ implies $f\in\BMO(X)$.
\end{proof}

The proof of the metric space version of the following John-Nirenberg
lemma can be found for example in Theorem 3.15 in
\cite{Bjornsbook}. See also \cite{Buckley} and \cite{MMNO}.

\begin{lemma}\label{lemma:jn}
Let $B\subset X$ be a ball and $f\in\BMO(5B)$. Then for every $\lambda>0$
\[
\mu(\{x\in B\,:\,|f(x)-f_{B}|>\lambda\})\leq
2\mu(B)\exp\left(-\frac{A\lambda}{\|f\|_*}\right).
\]
The positive constant $A$ depends only on the doubling constant $c_{D}$.
\end{lemma}

We are ready for the proof of the main result of this chapter.

\begin{proof}[Proof of Theorem~\ref{theorem:construction}]
  The necessity part of the theorem is an immediate consequence of
  Lemma \ref{lemma:jn}. Fix $\lambda > 1$ and let $B$ be a ball. By
  \eqref{eqn:sum1}, there exists $j_{0}$ such that
\[
(f_{j_{0}})_{B}\geq \frac1{N}.
\] 
Thus, by Lemma~\ref{lemma:jn} and \eqref{eqn:E0}, we have
\[
\begin{split}
\frac{\mu(B\cap E_{j_{0}})}{\mu(B)} 
&\leq \frac{\mu(\{x\in B\,:\, |f_{j_{0}}(x)-(f_{j_{0}})_{B}|\geq 1/N\})}{\mu(B)}\\
&\leq 2 e^{-A/(N \|f\|_*)}\leq 2\exp\left(-\frac{A\lambda}{N c_{2}}\right)\leq c_{D}^{-4\lambda},
\end{split}
\]
if $c_{2}$ is chosen to be small enough. This completes the proof of
the necessity part of Theorem \ref{theorem:construction}.

Then we consider the sufficiency.
By \eqref{eqn:minEj}, we have
\[
\mu\left(\bigcap_{j=1}^{N}E_{j}\right)=0.
\]
Thus, if $\lambda>1$ is smaller than a given number, then the functions
\[
f_{j}=\frac{\chi_{E_{j}^{c}}}{\sum_{k=1}^{N}\chi_{E_{k}^{c}}}, \quad1\leq j\leq N,
\]
satisfy the desired properties (we denote the characteristic function
of a set $A$ by $\chi_A$). So we may assume that $\lambda$ is large
enough.

First, we assume that 
\begin{equation}\label{eqn:unitball}
E_{1},\ldots,E_{N}\subset B_{0}
\end{equation}
for some $B_{0}\in \mathcal B_{0}$.
We will inductively construct the sequences of $\BMO$ functions
 $\{f_{j,h}\}_{h=1}^{\infty}$, $1\leq j\leq N$, such that
\begin{equation}\label{eqn:sumLambda}
\sum_{j=1}^{N}f_{j,h}(x)=\lambda,
\end{equation}
\begin{equation}\label{eqn:lessLambda}
0\leq f_{j,h}(x)\leq \lambda,
\end{equation}
\begin{equation}\label{eqn:gCond}
f_{j,h}(x)\leq g_{j}(B) \text{ for every } x\in B,\text{ if  }B\in \mathcal B_{h},
\end{equation}
and
\begin{equation}\label{eqn:bmoC1}
\|f_{j,h}\|_*\leq c_{1}.
\end{equation}
If the functions $f_{j,h}$ above have been constructed, there exists a sequence $1\leq h_{1}<h_{2}<\ldots$ such that
$\{f_{j,h_{k}}\}_{k=1}^{\infty}$ converge weak* in $L^{\infty}$ as $k\to\infty$, since $\|f_{j,h}\|_\infty\leq \lambda$ by \eqref{eqn:lessLambda}.
We set
\[
f_{j}=\text{weak}^*-\lim_{k\rightarrow\infty}\frac{f_{j,h_{k}}}\lambda, \quad1\leq j\leq N.
\]
Then \eqref{eqn:sum1} and \eqref{eqn:leq1} follow from
\eqref{eqn:sumLambda} and \eqref{eqn:lessLambda}. Let $g$ be as in
Lemma \ref{lemma:coifman-weiss}. Then
\[
\left|\int f_{j}g\,d\mu\right|
=\frac1\lambda\left|\lim_{k\rightarrow \infty}\int f_{j,h_{k}}g \,d\mu\right|
\leq \frac1\lambda\limsup_{k\rightarrow\infty}\|f_{j,h_{k}}\|_*
\leq \frac{c_{1}}\lambda.
\]
Thus \eqref{eqn:bmoF} with constant $2c_{1}$ follows from Lemma
\ref{lemma:coifman-weiss}. Since, by Lebesgue's theorem, 
\[
\lim_{r\to0}\sup_{\substack{B\ni x \\ \rad(B)\le r}} g_{j}(B)=0
\]
for $\mu$-almost every $x\in E_{j}$, we have by \eqref{eqn:gCond}
\[
\lim_{h\rightarrow 0}f_{j,h}(x)=0
\]
for $\mu$-almost every $x\in E_{j}$. Thus \eqref{eqn:E0} follows. Hence
$\{f_{j}\}_{j=1}^N$ are the desired functions.

To remove the restriction \eqref{eqn:unitball}, we take balls
$B_{p}\in\mathcal B_{-p}$, $p=1,2,\ldots$, such that $B_{p-1}\subset
B_{p}$ for every $p$, and we can construct $f_{j,p}$ such that all
other conditions are as for $B_{0}$, except that
\[
f_{j,p}=0\quad\text{on }E_{j}\cap B_{p}.
\]
Then there exists a sequence $1\leq p_{1}<p_{2}\ldots$ such that $\{f_{j,p_{k}}\}_{k=1}^{\infty}$ converge weak* in $L^{\infty}$. Then
\[
f_{j}=\text{weak}^{*}-\lim_{k\rightarrow\infty}f_{j,p_{k}},\quad1\leq j\leq N,
\]
are the desired functions.

Thus, to complete the proof Theorem~\ref{theorem:construction} we
shall construct a sequence of functions that satisfy the conditions
\eqref{eqn:sumLambda}--\eqref{eqn:bmoC1}. The proof is organized as
follows. In Lemma~\ref{lemma:easycond}, we will construct the sequence
$\{f_{j,h}\}_{h=0}^{\infty}$, $1\leq j\leq N$, and show that these
functions satisfy the conditions
\eqref{eqn:sumLambda}--\eqref{eqn:gCond}. And finally, in
Lemma~\ref{lemma:bmo}, we show that the condition
\eqref{eqn:bmoC1} is valid for the functions.

\begin{lemma}\label{lemma:easycond}
  Let $E_{1},\ldots,E_{N}$ satisfy \eqref{eqn:minEj} and
  \eqref{eqn:unitball}. Then there exist $\{f_{j,h}\}$ and
  $A_{j,h}\subset \mathcal B_{h}$ having the properties
  \eqref{eqn:sumLambda}--\eqref{eqn:gCond} and satisfying the
  following conditions
\begin{equation}\label{eqn:Lip}
|f_{j,h}(x)-f_{j,h}(y)|\leq 2^{(h+1)q} d(x,y),
\end{equation}
\begin{equation}\label{eqn:Ajh}
A_{j,h}=\{B\in\mathcal B_{h}:\, \sup_{B} f_{j,h-1}>g_{j}(B)\},
\end{equation}
\begin{equation}\label{eqn:cd3q}
f_{j,h}(x)\geq f_{j,h-1}(x)-c_{D}^{3}q,
\end{equation}
and
\begin{equation}\label{eqn:incr}
f_{j,h}(x)\geq f_{j,h-1}(x)\ \textrm{ for }\ x\notin \bigcup_{B\in A_{j,h}}2B.
\end{equation}
\end{lemma}

\begin{proof}
By \eqref{eqn:minEj}, we have
\[
\max_{1\leq j\leq N}g_{j}(B_{0})\geq 4\lambda.
\]
Set 
\[
s(B_{0})=\min\{j\,:\, 1\leq j\leq N, g_{j}(4B_{0})\geq 4\lambda\},
\]
\[
f_{s(B_{0}),0}=\lambda,
\quad\text{and}\quad
f_{j,0}=0\text{ for } j\neq s(B_{0}).
\]
Assume now that the functions $f_{1,k-1},\ldots,f_{N,k-1}$ have been
defined and satisfy the conditions
\eqref{eqn:sumLambda}--\eqref{eqn:gCond}, \eqref{eqn:Lip},
\eqref{eqn:cd3q} and \eqref{eqn:incr}. Define $A_{j,k}$ by
\eqref{eqn:Ajh}.  For any ball $B$, let $b_{B}$ denote a function that
is adapted to $B$, $0\leq b_{B}\leq 1$ and $b_{B}=1$ on $B$.  Let
$A_{j,k}=\{B_{m}\}_{m=1}^{p}$. Set
$a_{B_{1}}=\min\{qb_{B_{1}},f_{j,k-1}\}$ and
\[
a_{B_{m}}=\min\left\{qb_{B_{m}},f_{j,k-1}-\sum_{n=1}^{m-1}a_{B_{n}}\right\} \text{ for } m=2,\ldots,p.
\]
Since the supports of $\{b_{B_{m}}\}$ overlap at most $c_{D}^{3}$
times, the functions $c_{D}^{-3}q^{-1}a_{B_{m}}$ are adapted to
$B_{m}$.  Set
\[
\widetilde f_{j,k}=f_{j,k-1}-\sum_{B\in A_{j,k}}a_{B}=f_{j,k-1}-v_{j,k}.
\]
Since 
\[
\widetilde f_{j,k}=\max\left\{f_{j,k-1}-\sum_{B\in A_{j,k}}qb_{B},0\right\},
\]
we see that $\{\widetilde f_{j,k}\}$ satisfy \eqref{eqn:lessLambda}, \eqref{eqn:cd3q} and  \eqref{eqn:incr}.

If  $B\in A_{j,k}$ and $x\in B$, then  by Lemma \ref{lemma:g}
\[
\widetilde f_{j,k}(x)\leq \max\{f_{j,k-1}(x)-q,0\}
\leq\max\{g_{j}(\widetilde B)-q,0\}
\leq g_{j}(B),
\]
for every $\widetilde B \in \mathcal B_{k-1}$ such that $B\subset
\widetilde B$.

If $B\in \mathcal B_{k}\setminus A_{j,k}$ and $x\in B$, then
\[
\widetilde f_{j,k}(x)\leq f_{j,k-1}(x)\leq g_{j}(B)
\]  
by the definition of $A_{j,k}$. So $\{\widetilde f_{j,k}\}$ satisfies
\eqref{eqn:gCond}. These functions do not satisfy the property
\eqref{eqn:sumLambda}, and hence we shall modify the functions
further. We set
\[
f_{j,k}=\widetilde f_{j,k}+\sum_{\substack{B\in \bigcup_{m=1}^{N}A_{m,k} \\ s(B)=j}}a_{B}
=\widetilde f_{j,k} +w_{j,k}.
\]
The modified sequence $\{ f_{j,k}\}$ satisfies
\eqref{eqn:sumLambda}. Also the conditions \eqref{eqn:lessLambda},
\eqref{eqn:cd3q}, and \eqref{eqn:incr} are met since $a_{B}\geq 0$.

Let us next look at the condition \eqref{eqn:gCond}. If $B\in\mathcal
B_{k}$ and $w_{j,k}=0$ on $B$, then
\[
f_{j,k}=\widetilde f_{j,k}\leq g_{j}(B)\quad\text{on }B,
\]   
since $\widetilde f_{j,k}$ satisfies \eqref{eqn:gCond}. 
If $B\in \mathcal B_{k}$ and $w_{j,k}\neq 0$ on $B$, then, by the definition of $w_{j,k}$, there exists a ball $\widetilde B\in\mathcal B_{k}$ such that
\[
B\cap 2\widetilde B\neq \emptyset\quad\text{and}\quad g_{j}(4\widetilde B)\geq 4\lambda.
\] 
Then $B\subset 4\widetilde B$. By Lemma \ref{lemma:g},
\[
g_{j}(B)\geq g_{j}(4\widetilde B)-2\geq \lambda.
\]
So by \eqref{eqn:lessLambda}, we have
\[
f_{j,k}(x)\leq \lambda \leq g_{j}(B)
\]
and consequently \eqref{eqn:gCond} holds.

Let us show that the condition \eqref{eqn:Lip} holds. If $x,y\in
\widetilde B$ and $\widetilde B\in\mathcal B_{k}$, then
\begin{equation}\label{eqn:vwLip}
\begin{split}
|(-v_{j,k}(x)+w_{j,k}(x))&-(-v_{j,k}(y)+w_{j,k}(y))|\\
&\leq \sum_{B\in \bigcup_{m=1}^{N}A_{m,k} }|a_{B}(x)-a_{B}(y)|
\end{split}
\end{equation}
Since the supports of $\{ a_{B}\}_{B\in \bigcup_{m}A_{m,k}}$ overlap at most $Nc_{D}^{3}$ times, \eqref{eqn:vwLip} is dominated by
\[
Nc_{D}^{3} \cdot c_{D}^{3}q\cdot \frac{d(x,y)}{r_{k}}=Nc_{D}^{6}q 2^{qk}d(x,y).
\]
From this we conclude that
\[
\begin{split}
|f_{j,k}(x)-f_{j,k}(y)|&\leq |f_{j,k-1}(x)-f_{j,k-1}(y)|+ Nc_{D}^{6}q 2^{qk}d(x,y)\\
&\leq (1+Nc_{D}^{6}q)2^{kq}d(x,y)
\leq 2^{(k+1)q}d(x,y),
\end{split}
\]
where we used \eqref{eqn:Lip} for $f_{j,k-1}$, and also the inequality
\eqref{eqn:q}.
\end{proof}

\begin{lemma}\label{lemma:f}
\[
f_{j,h}(x)\leq g_{j}(B)- \frac1{3}\log_{2}\frac{r}{r_{h}} + 8\cdot2^{q}+ 6
\]
for every $x\in B=B(y,r)$ for any $B$ such that $r\leq 4r_{h}$.
\end{lemma}
\begin{proof}
  There are at most $c_{D}^{3}$ balls in $B_{1},\ldots,B_{k}$ with the
  centers in $\mathcal D_{h}$ such that $B_{i}\cap B\neq \emptyset$.
  Let
\[
\delta=\min_{1\leq i\leq k}g_{j}(B_{i})=g_{j}(B_{i_{0}}).
\]
By \eqref{eqn:gCond}
\[
\inf_{x\in B}f_{j,h}(x)\leq \delta,
\]
and by \eqref{eqn:Lip} we have
\[
f_{j,h}(x)\leq \delta+2^{(h+1)q}2r\leq \delta + 8\cdot 2^{q}
\]
whenever $x\in B$.

On the other hand,
\[
\begin{split}
g_{j}(B)&=\log_{c_{D}}\frac{\mu(B)}{\mu(B\cap E_{j})} \\
&\geq \log_{c_{D}}\frac{\mu(B)}{\sum_{i}\mu(B_{i}\cap E_{j})} \\
&\geq  \log_{c_{D}}\frac{\mu(B)}{c_{D}^{3}\max_{i}\{\mu(B_{i}\cap E_{j})\}} \\
&=\log_{c_{D}}\frac{\mu(B)}{\mu(B_{i_{0}})}+
\log_{c_{D}}\frac{ \mu(B_{i_{0}}) }{ \mu(B_{i_{0}}\cap E_{j})}
+\log_{c_{D}}\frac{1}{c_{D}^{3}}\\
&\geq \log_{c_{D}}\frac{\mu(B)}{\mu(B_{i_{0}})}+\delta-3\\
&\geq \frac1{3}\log_{2}\frac{r}{r_{h}}+\delta-6.
\end{split}
\]
The desired result follows from the two previous estimates.
\end{proof}

We finish to proof of Theorem~\ref{theorem:construction} by proving
the following lemma.
\begin{lemma}\label{lemma:bmo}
 $\|f_{j,h}\|_*\leq c_{1}$.
\end{lemma}

\begin{proof}
Let $B=B(x,r)$ be any ball. If $r\leq 2^{-hq}$ then, by \eqref{eqn:Lip}, we have
\begin{equation}\label{eqn:bmoCsmall}
\inf_{c\in \mathbb R}\barint_{B}|f_{j,h}-c|d\mu\leq 2^{q}.
\end{equation}
If $0\leq n<h$ and $2^{-(n+1)q}<r\leq2^{-nq}$, let
\[
\beta_{j}=\barint_{B}f_{j,n}\,d\mu.
\]
Notice that by Lemma~\ref{lemma:f},
\begin{equation}\label{eqn:beta}
  \beta_{j}\leq g_{j}(4B)+\frac1{3}q+8\cdot 2^{q}+6.
\end{equation}
We will show that
\begin{equation}\label{eqn:bmoC}
\barint_{B}|f_{j,h}-\beta_{j}|\,d\mu\leq C.
\end{equation}
Let
\begin{equation}\label{eqn:GH}
\begin{split}
&\{x\in B\,:\, |f_{j,h}(x)-\beta_{j}|\geq\alpha\}\\
&=\{x\in B\,:\, f_{j,h}(x)<\beta_{j}-\alpha\}\cup\{x\in B\,:\, f_{j,h}(x)>\beta_{j}+\alpha\}\\
&=G(B,j,\alpha)\cup H(B,j,\alpha).
\end{split}
\end{equation}
First, we estimate $\mu(G(B,j,\alpha))$. Let $\alpha>2^{q+1}$. Note that $f_{j,n}(x)>\beta_{j}-2^{q+1}$ on $B$ by \eqref{eqn:Lip}. So if $x\in G(B,j,\alpha)$ then, by \eqref{eqn:incr}, there exists $\widetilde B\in A_{j,k}$, $n<k\leq h$, such that 
$x\in 2\widetilde B$ and $f_{j,k}(x)<\beta_{j}-\alpha$.
So by \eqref{eqn:cd3q}, we have
\[
f_{j,k-1}(x)<\beta_{j}-\alpha+c_{D}^{3}q,
\]
and by \eqref{eqn:Lip}
\[
f_{j,k-1}(y)<\beta_{j}-\alpha+c_{D}^{3}q+3
\]
for every $y\in \widetilde B$.
Thus, by the definition of $A_{j,k}$, we obtain
\[
g_{j}(\widetilde B)<\beta_{j}-\alpha+c_{D}^{3}q+3.
\]
By the above, we can use the standard $5$-covering theorem
(\cite[Lemma 1.7]{Bjornsbook}) and take disjoint balls
$\{B_{m}\}\subset \bigcup_{n<k\leq h}A_{j,k}$ such that
\[
B_{m}\subset 4B,\quad G(B,j,\alpha)\subset \bigcup_{m}5B_{m},
\]
and 
\begin{equation}\label{eqn:gjBm}
g_{j}(B_{m})<\beta_{j}-\alpha+c_{D}^{3}q+3.
\end{equation}
Thus
\begin{equation}\label{eqn:muG}
\begin{split}
\mu(G(B,j,\alpha))&\leq c_{D}^{3}\sum_{m}\mu(B_{m})=c_{D}^{3}\sum_{m}\mu(E_{j}\cap B_{m})c_{D}^{g_{j}(B_{m})}\\
&\leq C c_{D}^{\beta_{j}-\alpha}\sum_{m}\mu(E_{j}\cap B_{m})\\
&\leq C c_{D}^{g_{j}(4B)-\alpha}\sum_{m}\mu(E_{j}\cap B_{m})\\
&\leq C c_{D}^{g_{j}(4B)-\alpha}\mu(E_{j}\cap 4B)\leq C \mu(B) c_{D}^{-\alpha}.
\end{split}
\end{equation}
Here we used first \eqref{g}, then \eqref{eqn:gjBm}, \eqref{eqn:beta}
and finally \eqref{g} again.

Let us then estimate the measure $\mu(H(B,j,\alpha))$. Let
$\alpha>(N-1)2^{q+1}$. Note that $\sum_{m=1}^{N}\beta_{m}=\lambda$ by
\eqref{eqn:sumLambda}.  So if $x\in H(B,j,\alpha)$, then
\[
\begin{split}
\sum_{1\leq m\leq N,\, m\neq j}f_{m,h}(x)&=\lambda-f_{j,h}(x)
=\sum_{m=1}^{N}\beta_{m}-f_{j,h}(x) \\
&=\left(\sum_{1\leq m\leq N,\, m\neq j}\beta_{m}\right)-\left(f_{j,h}(x)-\beta_{j}\right)\\
&< \left(\sum_{1\leq m\leq N,\, m\neq j}\beta_{m} \right)-\alpha.
\end{split}
\]
Thus
\[
 \sum_{\substack{1\leq m\leq N \\ m\neq j}}(\beta_{m} -f_{m,h}(x))>\alpha.
\]
So
\[
x\in\bigcup_{\substack{1\leq m\leq N \\ m\neq j}}G(B,m,\alpha/(N-1)),
\]
and consequently
\[
H(B,j,\alpha)\subset\bigcup_{\substack{1\leq m\leq N \\ m\neq j}}G(B,m,\alpha/(N-1)).
\]

By \eqref{eqn:muG}, we have
\begin{equation}\label{eqn:muH}
\mu(H(B,j,\alpha))\leq C(N-1)\mu(B)c_{D}^{-\alpha/(N-1)}.
\end{equation}
Thus, if $2^{-hq}\leq r\leq 1$, then \eqref{eqn:bmoC} follows from
\eqref{eqn:muG} and \eqref{eqn:muH}.  If $r>1$, then put
$\beta_{s(B_{0})}=\lambda$ and $\beta_{j}=0$ for $j\neq
s(B_{0})$. Then \eqref{eqn:bmoC} follows from the same argument. Thus
Lemma \ref{lemma:bmo} follows from \eqref{eqn:bmoCsmall} and
\eqref{eqn:bmoC}.
\end{proof}
The proof of Theorem~\ref{theorem:construction} is now complete.
\end{proof}

\section{Characterizations of $\BMO$-maps}
\label{sect:G}

We say that a $\mu$-measurable map $F\colon X \to X$ is a {\it
  $\BMO$-map} if 
\begin{itemize}
\item[(I)] $F^{-1}(E)$ is a $\mu$-null set for each $\mu$-null
set $E\subset X$, 
\item[(II)] for every $f\in\BMO(X)$ the composed map $C_F(f)=f\circ F$ is
  in $\BMO(X)$.
\end{itemize}

We shall prove a metric space generalization of a theorem due to
Gotoh~\cite[Theorem 3.1]{Gotoh01} which characterizes $\BMO$-maps
between doubling metric measure spaces. In the proof we apply
Uchiyama's construction proved in Section~\ref{sect:U}.
The condition \eqref{eq:i} has a similar flavor as the conditions in \cite{GK74} and \cite{KoMaSha12} 
related to invariance properties of quasiconformal mappings. 

\begin{theorem} \label{thm:Gotoh} Suppose that $F\colon X \to X$ is
  $\mu$-measurable. Then the following conditions are equivalent:
\begin{itemize}

\item[(i)] There exist positive finite constants $K$ and $\alpha$ such
  that for an arbitrary pair of $\mu$-measurable subsets $E_1,\, E_2$
  of $X$ we have
  \begin{equation} \label{eq:i}
    \sup_{B}\min_{k=1,2}\frac{\mu(F^{-1}(E_k)\cap B)}{\mu(B)}\leq
    K\left(\sup_{B}\min_{k=1,2}\frac{\mu(E_k\cap
        B)}{\mu(B)}\right)^\alpha,
\end{equation}
where the suprema are taken over all balls $B$ in $X$;

\item[(ii)] There exist constants $0<\gamma<1/4$ and $\lambda>0$ such
  that for an arbitrary pair of $\mu$-measurable subsets $E_1,\, E_2$
  of $X$ satisfying
\[
\sup_{B}\min_{k=1,2}\frac{\mu(E_k\cap B)}{\mu(B)}<\lambda,
\]
we have
\[
\sup_{B}\min_{k=1,2}\frac{\mu(F^{-1}(E_k)\cap B)}{\mu(B)}<\gamma,
\]
where the suprema are taken over all balls $B$ in $X$;

\item[(iii)] $F$ is a $\BMO$-map with the operator norm of $C_F$ bounded by $CK/\alpha$,
where $C$ depends only on the doubling constant.

\end{itemize}
\end{theorem}

The condition (i) readily implies the condition (ii), and hence to
show the equivalence of conditions (i)--(iii), it is enough to prove
implications (i)$\Rightarrow$(iii), (ii)$\Rightarrow$(iii) and
(iii)$\Rightarrow$(i), in Propositions~\ref{prop:i-iii},
\ref{prop:ii-iii}, and \ref{prop:iii-i}, respectively. The Uchiyama
construction of $\BMO$ functions, presented in Section~\ref{sect:U},
is used in the proof of Proposition~\ref{prop:iii-i}.
For the proof of the bound for the operator norm, see Proposition \ref{prop:normbound}.

\begin{remark} \label{rmk:conditions} Let us comment on the condition
  (i).
\begin{itemize}

\item[(1)] Setting $E_1=E_2=X$ in \eqref{eq:i} it can be seen that $K\geq 1$.

\item[(2)] If \eqref{eq:i} is valid for some positive $\alpha_0$ it
  clearly holds for all $0<\alpha<\alpha_0$. And moreover, since the
  condition \eqref{eq:i} is interesting mainly with small values of
  the exponent $\alpha$, we shall assume, without loss of generality,
  that $\alpha\leq 1$.

\end{itemize}
\end{remark}

We shall next prove several lemmas on $\BMO$ functions.
 
\begin{lemma} \label{lemma:JN2} Let $f\in\BMO(X)$. Then 
\begin{align*}
  & \min\{\mu(\{x\in B\colon f(x)\geq t\}),\,\mu(\{x\in B\colon f(x)\leq s\}) \} \\
  & \qquad \quad \leq 2\mu(B)\exp\left(-C\frac{t-s}{\|f\|_*}\right)
\end{align*}
for every $-\infty<s\leq t<\infty$, where $C$ is a positive
constant depending on the doubling constant $c_D$.
\end{lemma}

\begin{proof} By symmetry, we may assume that $f_B \leq (s+t)/2$. Then
  Lemma~\ref{lemma:jn} implies that
\begin{align*}
  \mu(\{x\in B\colon f(x)\geq t\}) & \leq \mu\left(\left\{x\in B\colon |f(x)-f_B|\geq \frac{t-s}{2}\right\}\right) \\
  & \leq 2\mu(B)\exp\left(-\frac{A(t-s)}{2\|f\|_*}\right).
\end{align*}
If $f_B \geq (s+t)/2$, we get a similar estimate for $\mu(\{x\in B\colon f(x)\leq s\})$.
\end{proof}

A converse of the statement in Lemma~\ref{lemma:JN2} is presented in
the following.

\begin{lemma} \label{lemma:JN2conv} Let $f\colon X\to\R$ be a
  $\mu$-measurable function with $|f|<\infty$ $\mu$-almost everywhere
  in $X$. Assume there exist positive constants $C_1,\,C_2$ such that
  for every ball $B$ in $X$ we have
\begin{align*}
  & \min\{\mu(\{x\in B\colon f(x)\geq t\}),\,\mu(\{x\in B\colon f(x)\leq s\}) \} \\
  & \qquad \quad \leq C_1\mu(B)\exp\left(-C_2(t-s)\right)
\end{align*}
for every $-\infty<s\leq t<\infty$. Then $f\in\BMO(X)$ and
\[
\|f\|_*\leq 4(C_1+1)C_2^{-1}\exp(2C_2).
\]
\end{lemma}

In the proof of Lemma~\ref{lemma:JN2conv} we apply the following lemma which can
be found in \cite[Lemma 4.5]{Gotoh01}.

\begin{lemma} \label{lemma:Gotoh}
  Let $\lambda\colon\R\to[0,1]$ be a non-constant, non-decreasing
  function. Assume that there exists positive constants $C_1,\,C_2$
  such that
\[
\min\{\lambda(s),1-\lambda(t)\} \leq C_1\exp(-C_2(t-s))
\]
for every $-\infty<s\leq t<\infty$. Then there exists $t_0\in\R$ such that
\[
\max\{\lambda(t_0-t),1-\lambda(t_0+t)\} \leq (C_1+1)\exp(2C_2)\exp(-C_2t)
\]
for each $t\geq 0$.
\end{lemma}

\begin{proof}[Proof of Lemma~~\ref{lemma:JN2conv}]
  We apply Lemma~\ref{lemma:Gotoh} by setting 
\[
\lambda(t) = \frac{\mu(\{x\in
  B\colon f(x)\leq t\})}{\mu(B)}.
\]
Then by the hypothesis $\lambda(t)$ meets the assumption in
Lemma~\ref{lemma:Gotoh} with the same constants $C_1$ and $C_2$. Hence
there exists $t_0\in\R$ such that the second inequality of
Lemma~\ref{lemma:Gotoh} is valid for every $t\geq 0$. This implies that
\[
\nu(t)= \mu(\{x\in B\colon |f(x)-t_0|\geq t\}) \leq
2(C_1+1)\mu(B)\exp(2C_2)\exp(-C_{2}t)
\]
for every $t\geq 0$. We obtain
\begin{align*}
\int_B|f-f_B|\, d\mu & \leq 2\int_B|f-t_0|\, d\mu =
2\int_0^\infty\nu(t)\, dt \\
& \leq 4(C_1+1)C_2^{-1}\exp(2C_2)\mu(B)
\end{align*}
from which the claim follows.
\end{proof}

In Euclidean spaces the following lemma is due to Str\"omberg
\cite{Str}. A result similar to this has also been considered for
nondoubling measures by Lerner in \cite{Lerner}.

\begin{lemma} \label{lemma:Stromberg} Let $f\colon X\to\R$ be
  $\mu$-measurable. Assume that there exist constants
  $0<\gamma<(4c_D^{3})^{-1}$, and $\lambda>0$ such that for each ball
  $B$ in $X$ we have
\begin{equation}\label{lambdagamma}
\inf_{c\in\R}\mu(\{x\in B\colon |f(x)-c|\geq \lambda\}) \leq \gamma\mu(B).
\end{equation}
Then $f\in\BMO(X)$ satisfying $\|f\|_* \leq C\lambda$, where a
positive constant $C$ depends only on the doubling constant $c_D$.
\end{lemma}

\begin{proof}
  Let $f$ be $\mu$-measurable on $X$, and fix $\gamma$ and $\lambda$
  such that the hypothesis \eqref{lambdagamma} is satisfied for each
  ball in $X$. Fix a ball $B\subset X$ and let $c_{0}$ be the number
  where the infimum in \eqref{lambdagamma} is reached.  For each
  $m=1,2,\ldots$ we write
\begin{align*}
S_{m}^{+} & =  \{x\in B\colon f(x)-c_{0}> m\lambda\}, \\
S_{m}^{-} & =  \{x\in B\colon f(x)-c_{0}< -m\lambda\}, \\
S_m & = S_{m}^{+}\cup S_{m}^{-}=\{x\in B\colon |f(x)-c_{0}|> m\lambda\}, \\
E_m & = \{x\in B\colon m\lambda <|f(x)-c_{0}|\leq (m+1)\lambda\},
\end{align*}
and 
\[
E_0 = \{x\in B\colon |f(x)-c_{0}|\leq \lambda\}.
\] 
Let us estimate the measure of the set $S_{m}^{+}$. First notice that
$S_{m}^{+}\subset S_{m-1}^{+}$.  For $\mu$-almost every $x\in
S_{m-1}^{+}$, there exists a ball $B_{x}=B(x,r_{x})$ such that
\begin{equation}\label{eqn:half}
\frac{1}{2c_{D}}\mu(B_{x})<\mu(B_{x}\cap S_{m-1}^{+})\leq \frac 12\mu(B_{x})
\end{equation}
and 
\[
\mu(B(x,r)\cap S_{m-1}^{+})>\frac12\mu(B(x,r))
\]
for all $r<\tfrac12r_{x}$; see, for example, Theorem 3.1 and Remark 3.2
in \cite{KKST}.

By a well known 5-covering theorem (\cite[Lemma 1.7]{Bjornsbook}), we
can cover the set $S_{m-1}^{+}$ by finite or countable sequence of
balls $\{B_{i}\}_{i}$ satisfying \eqref{eqn:half} such that the balls
$\{\tfrac15 B_{i}\}_{i}$ are disjoint. It follows from
\eqref{eqn:half} that the infimum in \eqref{lambdagamma} is reached
with some constant $c$ such that $$c_0+(m-2)\lambda\leq c \leq c_0+m\lambda$$
in each of the balls $B_{i}$, and hence $c-c_0\leq m\lambda$.

We conclude, by applying the in inequality \eqref{lambdagamma} in
balls $B_{i}$, that
\[
\begin{split}
  \mu(S_{m+1}^{+}) & \leq \sum_{i}\mu(B_i\cap S_{m+1}^{+})
  \leq\gamma\sum_{i}\mu(B_{i}) \leq c_{D}^{3}\gamma\sum_{i}\mu(\tfrac15 B_{i}) \\
  & \leq 2c_{D}^{3}\gamma\sum_{i}\mu(\tfrac15 B_{i}\cap S_{m-1}^{+})
  \leq 2c_{D}^{3}\gamma\mu(S_{m-1}^{+})
\end{split}
\]
Since $\mu(S_{1}^{+})\leq \mu(S_1)<\gamma\mu(B)$, it follows from the
previous estimate that
\[
\mu(S_{2m+2}^{+})\leq \mu(S_{2m+1}^{+}) \leq (2c_{D}^{3}\gamma)^{m+1}\mu(B)
\]
for each $m=1,2,\ldots$. Since a similar estimate holds for
$S_{m}^{-}$, we altogether have
\[
\mu(S_{m})\leq 2(2c_{D}^{3}\gamma)^{m/2}\mu(B).
\]
We thus conclude
  \begin{align*}
    \barint_{B}|f-f_B|\, d\mu & \leq
    \frac2{\mu(B)}\left(\sum_{m=0}^\infty\int_{E_m}|f-c_{0}|\, d\mu\right) \\
    & \leq \lambda +
    2\sum_{m=1}^\infty(m+1)\lambda\frac{\mu(S_m)}{\mu(B)} \\
    & \leq
    \lambda\left(1+2\sum_{m=1}^\infty(m+1)(2c_{D}^{3}\gamma)^{m/2}\right) \\
    & \leq
    \lambda\left(1+2\sum_{m=1}^\infty(m+1)2^{-m/2}\right).
\end{align*}
Since the preceding estimate holds for any ball $B\subset X$, the
claim follows.
\end{proof}

Let us now turn to the proof of Theorem~\ref{thm:Gotoh}.

\begin{proposition}\label{prop:normbound}[(i) $\Rightarrow$ (iii)] \label{prop:i-iii} Let
  $F\colon X\to X$ be $\mu$-measurable and assume that there exist
  positive finite constants $K$ and $\alpha$ such that the condition
  (i) of Theorem~\ref{thm:Gotoh} holds. Then $F$ is a $\BMO$-map
  satisfying $\|C_F\|\leq CK/\alpha$, where $C$ depends on the
  doubling constant $c_D$.
\end{proposition}

\begin{proof}
  The condition (i) implies that if $E$ is a $\mu$-null subset of $X$
  then also $\mu(F^{-1}(E))= 0$.

  Let $f\in\BMO(X)$ and set for each $-\infty<s\leq t<\infty$
\begin{equation} \label{eq:sets}
E_1 = \{x\in X\colon f(x)\leq s\} \quad\textrm{and}\quad E_2 = \{x\in X\colon f(x)\geq
t\}.
\end{equation}
It follows from Lemma~\ref{lemma:JN2} that
\[
\min\{\mu(E_1\cap B),\,\mu(E_2\cap B)\} \leq
2\mu(B)\exp\left(-C\frac{t-s}{\|f\|_*}\right)
\]
for all balls $B$ in $X$. The condition (i) implies
\[
\min\{\mu(F^{-1}(E_1)\cap B),\,\mu(F^{-1}(E_2)\cap B)\} \leq
2^\alpha K\mu(B)\exp\left(-C\frac{\alpha(t-s)}{\|f\|_*}\right)
\]
for all balls $B$ in $X$. Since
\[
  F^{-1}(E_1)\cap B = \{x\in B\colon (f\circ F)(x)\leq s\}
\]
and 
\[
F^{-1}(E_2)\cap B = \{x\in B\colon (f\circ F)(x)\geq t\},
\]
it follows from Lemma~\ref{lemma:JN2conv} that $f\circ F\in \BMO(X)$
and (recall that $\alpha\leq 1$, see Remark~\ref{rmk:conditions})
\begin{align*}
  \|C_F(f)\|_* & \leq
  \frac{4(2^{\alpha}K+1)\|f\|_*}{C\alpha}\exp(2C\alpha/\|f\|_*) \\
  & =\frac{CK\|f\|_*}{\alpha}\exp(C\alpha/\|f\|_*),
\end{align*}
where $C$ is a positive constant depending on the doubling constant
$c_D$. Applying the preceding estimate to $\tau f$, $\tau>0$, and
letting $\tau\to\infty$, we obtain that $\|C_F\|\leq CK/\alpha$.
\end{proof}

\begin{proposition}[(ii) $\Rightarrow$ (iii)] \label{prop:ii-iii} Let
  $F\colon X\to X$ be $\mu$-measurable and assume that there exist
  constants $0<\gamma< (4c_D^3)^{-1}$ and $\lambda>0$ such that the
  condition (ii) of Theorem~\ref{thm:Gotoh} holds. Then $F$ is a
  $\BMO$-map satisfying $\|C_F\|\leq C\lambda$, where $C$ depends on
  the doubling constant $c_D$ and $\gamma$.
\end{proposition}

\begin{proof}
  The condition (ii) implies that if $E$ is a $\mu$-null subset of $X$
  then also $\mu(F^{-1}(E))= 0$.

  Let $f\in\BMO(X)$ and assume, without loss of generality, that
  $\|f\|_*=1$. We define the sets $E_1$ and $E_2$ for each
  $-\infty<s<t<\infty$ as in \eqref{eq:sets}. We apply
  Lemma~\ref{lemma:JN2} and obtain
\[
\sup_{B}\min_{k=1,2}\frac{\mu(E_k\cap B)}{\mu(B)} \leq
2\exp\left(-C(t-s)\right) <\lambda,
\]
whenever $t-s\geq C_1$, where $C_1$ only depends on $\lambda$ and the
constant $C$ from Lemma~\ref{lemma:JN2}. Hence the condition (ii)
implies that
\[
\sup_{B}\min_{k=1,2}\frac{\mu(F^{-1}(E_k)\cap B)}{\mu(B)}<\gamma.
\]
For every ball $B$ in $X$ we set
\begin{align*}
s_B = \sup & \left\{s\in\R:\, \mu(\{x\in B\colon f(F(x))\leq s\}) \right. \\
& \qquad \qquad \left. \leq
  \mu(\{x\in B\colon f(F(x))\geq s+C_1\})\right\}.
\end{align*}
Since $|f(F(x))|<\infty$ for $\mu$-almost every $x\in X$,
we have that $s_B\neq \pm\infty$. Hence
\[
\mu(\{x\in B\colon f(F(x))\leq s_B-1\}) < \gamma\mu(B)
\]
and
\[
\mu(\{x\in B\colon f(F(x))\geq s_B+C_1+1\}) < \gamma\mu(B).
\]
If we set $c_B = s_B+C_1/2$ and $\tau=1+C_1/2$, we obtain
\[
\mu(\{x\in B\colon |f(F(x))-c_B| \geq \tau\}) \leq 2\gamma\mu(B).
\]
The claim follows from Lemma~\ref{lemma:Stromberg}.
\end{proof}

We shall apply the Uchiyama construction in the proof of the following
result.

\begin{proposition}[(iii) $\Rightarrow$ (i)] \label{prop:iii-i} Let
  $F\colon X\to X$ be a $\BMO$-map. Then there exist positive
  constants $K$ and $\beta$, depending only on the doubling constant
  $c_D$, such that the condition (i) of Theorem~\ref{thm:Gotoh} holds
  with $\alpha=\beta/\|C_F\|$.
\end{proposition}

\begin{proof}
  Let $E_1$ and $E_2$ be $\mu$-measurable subsets in $X$ and let 
  $\lambda>0$ be such that
\[
c_D^{-4\lambda} = \sup_{B}\min_{k=1,2}\frac{\mu(E_k\cap B)}{\mu(B)}.
\]
By Theorem~\ref{theorem:construction} there exist the functions $f_1$
and $f_2$, both in $\BMO(X)$, such that $f_1+f_2=1$, $0\leq f_k\leq
1$, $f_k=0$ on $E_k$, and $\|f_k\|_*\leq C_1/\lambda$ for $k=1,2$,
where a positive constant $C_1$ depends on the doubling constant
$c_D$. Define for $k=1,2$ the composed function
$g_k = f_k\circ F$.
Then $g_1+g_2=1$, $0\leq g_k\leq 1$, $g_k=0$ on $F^{-1}(E_k)$, and
$\|g_k\|_*\leq C_1\|C_F\|/\lambda$ for $k=1,2$.

Let us fix a ball $B$ in $X$. Clearly, we may assume that $(g_1)_B\geq
1/2$. Then by Lemma~\ref{lemma:jn} we obtain
\begin{align*}
  \frac{\mu(F^{-1}(E_1)\cap B)}{\mu(B)} & \leq \frac{\mu(\{x\in B\colon |g_1(x)-(g_1)_B|\geq
  1/2\})}{\mu(B)} \\
& \leq 2\exp(-C\lambda/\|C_F\|),
\end{align*}
where $C$ is a positive constant depending on the doubling constant
$c_D$. By plugging in the value of $\lambda$, we obtain 
\begin{align*}
\sup_{B}\min_{k=1,2}\frac{\mu(F^{-1}(E_k)\cap B)}{\mu(B)} & \leq 
2\left(\sup_{B}\min_{k=1,2}\frac{\mu(E_k\cap B)}{\mu(B)}\right)^{C/\|C_F\|}
\end{align*}
which completes the proof.
\end{proof}

\subsection{$A_p$-weights and $\BMO$-maps}

We close this paper by discussing the connection between Muckenhoupt
$A_p$-weights and $\BMO$-maps.

It is well known that if $\omega$ is an $A_p$-weight for some $1\leq
p<\infty$, then $\log\omega \in \BMO(X)$, and on the other hand,
whenever $f\in \BMO(X)$, then $e^{\delta f}$ is an $A_p$-weight for
some $\delta>0$ and $1\leq p<\infty$. We refer to \cite{GCRdF} for
this result in the Euclidean setting. It straightforward to verify
that the result has its counterpart also in metric measure spaces with
a doubling measure. 

We can add the following condition to the list in
Theorem~\ref{thm:Gotoh}:
\begin{itemize}

\item[(iv)] For each $A_p$-weight $\omega$, with some $1\leq
  p<\infty$, the composed map $\omega^\delta\circ F$ is an
  $A_{p'}$-weight for some positive $\delta$ and $1\leq p'<\infty$.

\end{itemize}

In Euclidean spaces, the condition (iv) can be stated in terms of
$A_\infty$-weights, see \cite[Corollary 3.3]{Gotoh05}, and these
weights have several but equivalent characterizations. In general
metric spaces $A_\infty$-weights have first been defined and studied
in \cite{StrTor}. In this generality, however, these different
conditions are not necessarily equivalent. In particular, the class of
$A_\infty$-weights can be strictly larger than the union of
$A_p$-weights \cite{StrTor}. Several characterizations for
$A_\infty$-weights and their relations in doubling metric measure
spaces have also been studied in \cite{KoKa}.




\noindent (J.K.): Aalto University, School of Science and Technology, Department of Mathematics, P.O. Box 11100, FI-00076 Aalto, Finland. \\ 
E-mail: {\tt juha.k.kinnunen@aalto.fi}

\medskip

\noindent (R.K.), (N.M): Department of Mathematics and Statistics, P.O. Box 68, FI-00014 University of Helsinki, Finland. \\
E-mail: {\tt riikka.korte@helsinki.fi}, {\tt niko.marola@helsinki.fi} 

\medskip

\noindent (N.S.): Department of Mathematical Sciences, P.O.Box 210025, University of Cincinnati, Cincinnati, OH 45221{0025, U.S.A. \\
E-mail:{\tt shanmun@uc.edu}

\end{document}